\documentclass[11pt,reqno]{amsart}
\usepackage[margin=1.08in]{geometry}
\usepackage{amsmath,amssymb,amsthm}
\usepackage{enumitem}
\usepackage{hyperref}
\usepackage{tikz}
\usetikzlibrary{arrows.meta}
\hypersetup{colorlinks=true,linkcolor=blue,citecolor=blue,urlcolor=blue}

\theoremstyle{plain}
\newtheorem{theorem}{Theorem}[section]
\newtheorem{lemma}[theorem]{Lemma}
\newtheorem{proposition}[theorem]{Proposition}
\newtheorem{corollary}[theorem]{Corollary}
\newtheorem{conjecture}[theorem]{Conjecture}

\theoremstyle{definition}
\newtheorem{definition}[theorem]{Definition}
\newtheorem{remark}[theorem]{Remark}

\newcommand{\R}{\mathbb{R}}
\newcommand{\C}{\mathbb{C}}

\DeclareMathOperator{\RP}{\mathbb{RP}}

\title[Visibility cliques and dense orchard cores]{Visibility cliques, cubic containers,\\ and dense orchard cores}

\author{Sohail (Neel) Sarkar}
\thanks{University of Toronto. Email: \texttt{sohail.sarkar@mail.utoronto.ca}.}

\date{April 30, 2026}

\begin{document}

\begin{abstract}
The Big-Line--Big-Clique Conjecture of K\'ara, P\'or and Wood asserts that, for every fixed $k$ and $\ell$, every sufficiently large finite planar point set contains either $k$ collinear points or $\ell$ pairwise visible points. We prove a quantitative form in two structured regimes and isolate a structural obstruction to the full conjecture.

The main result is a deterministic cubic-container theorem. If $A\subset\R^2$ has $n$ points, no $k$ collinear points, and all but $s$ points of $A$ lie on a real cubic, then the cubic-supported part of $A$ has a visible clique cover of size $O_k(s+1)$; in particular $V(A)$ contains a clique of size $\Omega_k(n/(s+1))$, unless the cubic is the excluded three-line case containing only $O_k(1)$ points. Combining this with the Green--Tao structure theorem, we obtain that every $n$-point set with no $k$ collinear points and at most $Kn$ ordinary lines contains a visible clique of size $\Omega_{k,K}(n)$; more strongly, all but $O_K(1)$ points can be partitioned into $O_{k,K}(1)$ mutually visible sets.

We also combine the cubic-container theorem with the Elekes--Szab\'o theorem on triple lines and cubic curves to prove the Big-Line--Big-Clique conclusion for point sets contained in any fixed irreducible algebraic curve. Finally, we prove a dense-orchard core lemma showing that the absence of a visible $K_\ell$ forces a positive-density subset in which every point lies on linearly many $3$-rich lines, and we give a sharp one-blocker example showing why ambient blockers cannot be ignored.
\end{abstract}

\maketitle

\section{Introduction}

Let $A\subset\R^2$ be finite. Two points $p,q\in A$ are \emph{visible in $A$} if the open segment $(pq)^\circ$ contains no other point of $A$. The \emph{visibility graph} $V(A)$ has vertex set $A$ and edges the visible pairs. A mutually visible subset of $A$ is precisely a clique of $V(A)$.

K\'ara, P\'or and Wood~\cite{KPW2005} posed the following Ramsey-type problem, now commonly called the Big-Line--Big-Clique Conjecture.

\begin{conjecture}[Big-Line--Big-Clique Conjecture]\label{conj:BLBC}
For every pair of integers $k,\ell\ge 2$ there is $n_0(k,\ell)$ such that every set $A\subset\R^2$ with $|A|\ge n_0(k,\ell)$ contains either $k$ collinear points or $\ell$ mutually visible points.
\end{conjecture}

The conjecture is trivial for $\ell\le3$ and is known through the empty-pentagon method for $\ell\le5$ in the notation of Conjecture~\ref{conj:BLBC}~\cite{Abel2011}. The infinite analogue is false~\cite{PorWood2010}, and the stronger visible-island variant is false even for finite point sets~\cite{LNS2022}. The standard extremal-graph first step is insufficient: if $V(A)$ has no $K_\ell$, then Tur\'an's theorem gives many non-visible pairs, hence many collinear triples, but this does not by itself produce the algebraic structure needed to force a visible clique. The interplay between visibility and external blockers is studied systematically by P\'or and Wood~\cite{PorWoodBlockers}. Green's Problem~69 asks for a cubic-structure theorem from many collinear triples; Green's Problem~70 is the Big-Line--Big-Clique problem and explicitly notes this connection~\cite{GreenProblems}. The algebraic-incidence background around such cubic phenomena is surveyed in~\cite{deZeeuwSurvey}.  Our use of cubic curves is also informed by the recent incidence literature around polynomial and algebraic methods.  Sheffer's monograph~\cite{ShefferBook} gives a systematic account of polynomial-method techniques in incidence geometry; Raz, Sharir and de Zeeuw~\cite{RazSharirDeZeeuw2016} prove an Elekes--Szab\'o theorem for Cartesian-product incidences with group-like exceptional structure; and Solymosi~\cite{Solymosi2022} studies structural consequences of many collinear triples under additional combinatorial hypotheses.  The role of the present paper is complementary: once a cubic or low-degree algebraic container is available, we give a visibility-theoretic mechanism that controls ambient blockers.

This paper develops the visibility side of that cubic-structure route. The first contribution is a deterministic theorem that converts an ambient cubic container into visible cliques. Its proof is independent of the source of the cubic.

\begin{theorem}[Cubic-container theorem, numerical form]\label{thm:intro-cubic-informal}
Let $A\subset\R^2$ have $n$ points and no $k$ collinear points. Suppose a real cubic $\Gamma$ contains all but $s$ points of $A$. If $|A\cap\Gamma|>3(k-1)$, then the cubic-supported part $A\cap\Gamma$ has a visible clique cover in $A$ of size $O_k(s+1)$, and
\[
  \omega(V(A))\ge
  \max\left\{1,
  \min\left(
  \frac{n-s-(k-1)}{s+k},
  \frac{n-s-13}{15(s+1)}
  \right)\right\}.
\]
\end{theorem}

The formal and sharper version is Theorem~\ref{thm:cubic-container}. In the irreducible case, after deleting at most $13$ exceptional points, the cubic splits into at most $15$ visibility patches; on each patch, the only possible blockers are the $s$ off-cubic points. In the line-plus-conic case, the conic is one visibility patch and the line component contributes at most $k-1$ additional possible blockers. The local combinatorial principle is that if a patch has $b$ external blockers, then the non-visibility graph on the patch has maximum degree at most $b$, so the patch is coverable by at most $b+1$ visible cliques.

The second contribution is the few-ordinary-lines theorem. Recall that an ordinary line of $A$ is a line containing exactly two points of $A$.

\begin{theorem}[Few ordinary lines force a linear visible clique]\label{thm:intro-few-ordinary}
For every integer $k\ge2$ and every $K>0$ there are constants $c=c(k,K)>0$ and $n_0=n_0(k,K)$ such that the following holds. If $A\subset\R^2$ has $n\ge n_0$ points, no $k$ collinear points, and at most $Kn$ ordinary lines, then
\[
  \omega(V(A))\ge c n.
\]
More strongly, all but $O_K(1)$ points of $A$ can be partitioned into $O_{k,K}(1)$ mutually visible sets.
\end{theorem}

This follows by applying the Green--Tao theorem that a sufficiently large set with at most $Kn$ ordinary lines has all but $O(K)$ points on a cubic, and then using the cubic-container theorem.

The third contribution is a bounded-degree algebraic-curve case of the conjecture.

\begin{theorem}[Irreducible algebraic curves, fixed-degree form]\label{thm:intro-curve}
Fix $d,k,\ell$. There is $n_0(d,k,\ell)$ such that if $C\subset\R^2$ is an irreducible real algebraic curve of degree at most $d$ and $A\subset C(\R)$ has at least $n_0$ points, then $A$ contains either $k$ collinear points or $\ell$ mutually visible points.
\end{theorem}

The proof combines Tur\'an's theorem with the Elekes--Szab\'o theorem: if a non-cubic irreducible algebraic curve supports quadratically many distinct triple lines, then $n$ is bounded in terms of the degree and density. If the curve is a cubic, the cubic-container theorem applies directly.

The fourth contribution is a dense-orchard core lemma. We prove that if $A$ has no $k$ collinear points and no visible $K_\ell$, then a positive-density subset $A'\subset A$ remains after pruning such that every point of $A'$ is incident to linearly many $3$-rich lines of $A'$. This gives a natural combinatorial object that a future proof of the full conjecture must convert into an ambient cubic container.

The paper is organized as follows. Section~\ref{sec:prelim} records the Green--Tao input. Section~\ref{sec:blocker} proves the local blocker-colouring lemma. Sections~\ref{sec:cubic-lemma} and~\ref{sec:conic} prove the visibility geometry for irreducible cubics and conics. Section~\ref{sec:container-proof} proves the cubic-container theorem. Section~\ref{sec:few-ordinary-proof} proves the few-ordinary-lines theorem. Section~\ref{sec:k4l4} gives a direct $(4,4)$ calibration. Section~\ref{sec:bounded-degree} proves the bounded-degree irreducible-curve and robust low-degree-container cases using Elekes--Szab\'o. Section~\ref{sec:sharpness} gives a sharp one-blocker construction. Section~\ref{sec:orchard} proves the dense-orchard core lemma, applies it inside bounded-degree curves, and formulates the ambient cubic-container problem.

\section{Preliminaries and the Green--Tao input}\label{sec:prelim}

We work in the affine plane $\R^2$ and its projective closure $\RP^2$.  Since all point sets are affine, a projective line at infinity contains no point of $A$ unless explicitly stated otherwise.  For $p,q\in\R^2$, $(pq)^\circ$ denotes the open Euclidean segment joining them.

A real projective cubic is the zero locus in $\RP^2$ of a nonzero homogeneous polynomial of degree $3$.  Over $\R$, every cubic is, up to multiplicity, one of the following: an irreducible cubic; a real line together with an irreducible real conic; or a union of three real lines, not necessarily distinct.  We use B\'ezout's theorem in the standard form that two projective plane curves of degrees $d$ and $e$ with no common component have at most $de$ intersection points over $\C$, counted with multiplicity.

We need only the following consequence of the Green--Tao structure theorem for point sets with few ordinary lines.

\begin{theorem}[Green--Tao structure theorem, consequence of \cite{GreenTao2013}]\label{thm:GT}
There are an absolute constant $C>0$, an absolute constant $c_{\rm GT}>0$, and a function $n_{\rm GT}(K)$ with
\[
  n_{\rm GT}(K)\le \exp\exp(CK^C)
\]
such that the following holds.  Let $P\subset\RP^2$ be a finite set of $n\ge n_{\rm GT}(K)$ points spanning at most $Kn$ ordinary lines.  Then there is a real projective cubic $\Gamma\subset\RP^2$ such that
\[
  |P\setminus\Gamma|\le c_{\rm GT}K.
\]
\end{theorem}

\begin{proof}[Justification of the stated consequence]
Green and Tao's full structure theorem says that, after a projective transformation if necessary, $P$ differs in $O(K)$ points from one of three model configurations: points on a line, a B\"or\H{o}czky-type configuration on a conic together with a line, or a coset of a finite subgroup on an irreducible cubic.  Each model is contained in a real cubic.  Pulling that cubic back by the inverse projective transformation gives a real cubic in the original projective plane containing all but $O(K)$ points of $P$.  Absorbing the absolute implicit constant into $c_{\rm GT}$ gives the theorem.
\end{proof}

\section{A local blocker-colouring lemma}\label{sec:blocker}

For a subset $Y\subseteq A$, let $\vartheta_A(Y)$ denote the least number of parts in a partition of $Y$ into subsets that are mutually visible in the ambient set $A$. Equivalently, $\vartheta_A(Y)$ is the chromatic number of the non-visibility graph induced by $Y$, but the notation emphasizes that blockers may lie anywhere in $A$.

The next lemma is the main combinatorial mechanism. It is stronger than an averaging estimate for one large clique: it gives a clique cover of the whole patch.

\begin{lemma}[Local blocker-colouring lemma]\label{lem:blocker-colouring}
Let $A\subset\R^2$ be finite and let $V\subset\R^2$ be a set such that no line contains three distinct points of $V$.  Put $X=A\cap V$ and $m=|X|$.  Let
\[
B:=\{r\in A\setminus V:\text{ there exist }p,q\in X\text{ with }r\in(pq)^\circ\},
\qquad b:=|B|.
\]
Then
\[
  \vartheta_A(X)\le b+1.
\]
Consequently,
\[
  \omega(V(A))\ge \frac{m}{b+1}.
\]
\end{lemma}

\begin{proof}
Let $H$ be the graph on vertex set $X$ whose edges are the non-visible pairs in $A$.  Thus $pq\in E(H)$ if and only if $(pq)^\circ\cap A\ne\varnothing$.  Every edge of $H$ has at least one blocker in $B$; choose one such blocker and assign the edge to it.

Fix $r\in B$.  The edges assigned to $r$ form a matching.  Indeed, if two assigned edges shared a vertex, say $pq$ and $pq'$, then $r\in(pq)^\circ\cap(pq')^\circ$.  Hence $p,q,q'$ lie on the same line, contradicting the assumption that no line contains three points of $V$.

Therefore each $r\in B$ contributes at most one edge incident to any fixed vertex of $H$.  It follows that $\Delta(H)\le b$.  By the greedy colouring theorem, $H$ has a proper vertex-colouring with at most $b+1$ colours.  Each colour class is an independent set in $H$, which means that every pair of its points is visible in the ambient set $A$.  The largest colour class has size at least $m/(b+1)$.
\end{proof}

\section{The uniform cubic visibility lemma}\label{sec:cubic-lemma}

\begin{definition}\label{def:patch}
Let $C\subset\R^2$ be a real algebraic curve and $V\subset C$ a subset.
We say $V$ is a \emph{visibility patch of $C$} if for every pair of
distinct points $p,q\in V$ the open segment $(pq)^\circ$ contains no
point of $C$.
\end{definition}

The key quantitative fact is the following. We state it in its
projective form, which is what the proof requires.

\begin{lemma}[Uniform cubic visibility lemma]\label{lem:cubic-visibility}
Let $\bar C\subset\RP^2$ be an irreducible real projective cubic, and
fix an affine chart $\R^2\subset\RP^2$ with line at infinity $L_\infty$
such that $L_\infty \not\subset\bar C$. Then there exist a finite set
$E\subset\bar C(\R)$ with $|E|\le 13$ and open arcs
$V_1,\ldots,V_M\subset \bar C(\R)$ with $M\le 15$ such that
\[
\bar C(\R) = E \cup V_1\cup\cdots\cup V_M,
\]
the union is disjoint, each $V_i$ lies in the affine chart $\R^2$, and
each $V_i$ is a \emph{visibility patch of $\bar C$}: for every two
distinct $p,q\in V_i$,
\[
(pq)^\circ\cap \bar C(\R)=\varnothing.
\]
Moreover, no line meets any single $V_i$ in three points.
\end{lemma}

For orientation, Figure~\ref{fig:cubic-patches} illustrates the
decomposition in the standard affine chart for the cubic
$y^2=x^3-x$.

\begin{center}
\centering
\begin{tikzpicture}[x=1.35cm,y=0.58cm,>=Latex]
  \draw[->,gray!65] (-1.35,0) -- (2.68,0) node[below right] {$x$};
  \draw[->,gray!65] (0,-3.25) -- (0,3.35) node[above left] {$y$};

  \draw[line width=1.1pt,blue!65!black,domain=-1:0,samples=90,smooth]
    plot ({\x},{sqrt(\x*\x*\x-\x)});
  \draw[line width=1.1pt,teal!70!black,domain=-1:0,samples=90,smooth]
    plot ({\x},{-sqrt(\x*\x*\x-\x)});
  \draw[line width=1.1pt,orange!85!black,domain=1:1.4679,samples=70,smooth]
    plot ({\x},{sqrt(\x*\x*\x-\x)});
  \draw[line width=1.1pt,red!70!black,domain=1.4679:2.35,samples=90,smooth]
    plot ({\x},{sqrt(\x*\x*\x-\x)});
  \draw[line width=1.1pt,violet!75!black,domain=1:1.4679,samples=70,smooth]
    plot ({\x},{-sqrt(\x*\x*\x-\x)});
  \draw[line width=1.1pt,green!55!black,domain=1.4679:2.35,samples=90,smooth]
    plot ({\x},{-sqrt(\x*\x*\x-\x)});

  \node[blue!65!black,font=\scriptsize] at (-0.52,0.95) {$V_1$};
  \node[teal!70!black,font=\scriptsize] at (-0.52,-0.95) {$V_2$};
  \node[orange!85!black,font=\scriptsize] at (1.22,0.72) {$V_3$};
  \node[red!70!black,font=\scriptsize] at (1.93,2.2) {$V_4$};
  \node[violet!75!black,font=\scriptsize] at (1.22,-0.72) {$V_5$};
  \node[green!55!black,font=\scriptsize] at (1.93,-2.2) {$V_6$};

  \foreach \x/\y in {-1/0,0/0,1/0}
    \node[draw=black,fill=white,minimum size=5.1pt,inner sep=0pt,line width=0.55pt] at (\x,\y) {};
  \fill[black] (1.4679,1.3019) circle (1.75pt);
  \fill[black] (1.4679,-1.3019) circle (1.75pt);

  \node[font=\scriptsize,anchor=east] at (-1.07,-0.48) {$E_{\mathrm{vt}}$};
  \node[font=\scriptsize,anchor=west] at (1.52,1.45) {$E_{\mathrm{fl}}$};
  \node[font=\scriptsize,anchor=west] at (1.52,-1.45) {$E_{\mathrm{fl}}$};
  \draw[->,gray!70,thick] (2.35,3.27) -- (2.55,3.55);
  \draw[->,gray!70,thick] (2.35,-3.27) -- (2.55,-3.55);
  \node[font=\scriptsize,anchor=west] at (2.45,3.15) {$E_\infty$};
\end{tikzpicture}
\refstepcounter{figure}\label{fig:cubic-patches}
\smallskip

\begin{minipage}{0.88\textwidth}
\small
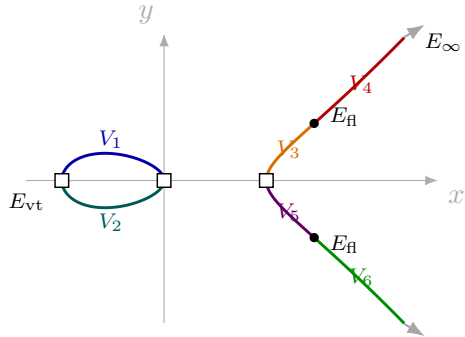
\textsc{Figure~\thefigure.} The cubic $y^2=x^3-x$ in its standard affine chart. The squares mark the three finite vertical tangencies $(-1,0)$, $(0,0)$ and $(1,0)$. The dots mark the two finite flexes, and the point at infinity $E_\infty=[0:1:0]$, indicated by the arrows, is the third real flex. Removing these exceptional points gives the labelled visibility patches. The labels $V_4,V_6$ denote the two unbounded affine arcs, which meet at the point at infinity in the projective closure. This standard chart is used only for visualization; the proof below first allows a generic affine change of coordinates.
\end{minipage}
\end{center}

We will prove the lemma in several steps.

\begin{proof}
Let $\bar F\in\R[X,Y,Z]$ be the homogenization of a defining polynomial
of $\bar C$. Fix coordinates so that $L_\infty=\{Z=0\}$, and choose a
further generic \emph{invertible affine change of coordinates} on $\R^2$
(i.e., a projective transformation of $\RP^2$ preserving $L_\infty$ as
a set), if necessary, so that the point $p_\infty:=[0:1:0]$
satisfies $p_\infty\notin\bar C$. Such affine changes preserve
collinearity, betweenness, and hence the visibility graph of any
finite set; they also preserve the set $\bar C(\R)$, the class of
real irreducible cubics, the points at infinity, vertical tangencies
(after relabeling the vertical direction), and the projective
Hessian. Thus the conclusion of the lemma (which concerns the
existence of a decomposition of $\bar C(\R)$ into points and
visibility patches with certain properties) is invariant under such
changes. This is a harmless genericity condition: $\bar C\cap L_\infty$
consists of at most $3$ points (B\'ezout), and we may choose the
affine rotation at infinity to avoid these three.

\smallskip
\textbf{Step 1: Define the exceptional set $E$ projectively.}
We let
\[
E \;:=\; E_{\mathrm{sing}}\;\cup\; E_\infty\;\cup\; E_{\mathrm{vt}}\;\cup\;
E_{\mathrm{fl}} \;\subset\; \bar C(\R),
\]
where
\begin{enumerate}
\item[(a)] $E_{\mathrm{sing}}$ is the set of singular points of
$\bar C$,
\item[(b)] $E_\infty := \bar C(\R)\cap L_\infty$,
\item[(c)] $E_{\mathrm{vt}}$ is the set of smooth points
$p\in\bar C(\R)\setminus E_\infty$ whose tangent line $T_p\bar C$
passes through the vertical infinity point $p_\infty=[0:1:0]$, and
\item[(d)] $E_{\mathrm{fl}}$ is the set of smooth points
$p\in\bar C(\R)\setminus E_\infty$ that are \emph{projective flexes}
of $\bar C$, i.e.\ where the projective Hessian $H_{\bar F}$ of the
homogeneous polynomial $\bar F$ vanishes.
\end{enumerate}
We bound each of the four sets.

\textit{(a) Singular points, $|E_{\mathrm{sing}}|\le 1$.} Suppose, for
contradiction, that $p\ne q$ are both singular points of $\bar C$, and
let $\ell$ be the projective line through $p$ and $q$. At any singular
point of $\bar C$, every line through that point meets $\bar C$ with
intersection multiplicity $\ge 2$, because the local equation of
$\bar C$ has vanishing linear part. Hence
$\mu_p(\ell\cap\bar C)\ge 2$ and $\mu_q(\ell\cap\bar C)\ge 2$, so
$\sum_{r\in\ell\cap\bar C}\mu_r\ge 4$. If $\ell\not\subset\bar C$,
B\'ezout gives $\sum_r\mu_r = \deg\ell\cdot\deg\bar C = 3$, a
contradiction. If $\ell\subset\bar C$, then $\bar C$ has $\ell$ as a
degree-$1$ component, contradicting the irreducibility of $\bar C$
(which has degree $3$). Hence $|E_{\mathrm{sing}}|\le 1$.

\textit{(b) Points at infinity, $|E_\infty|\le 3$.} Since
$L_\infty\not\subset\bar C$ and $\deg L_\infty=1$, B\'ezout gives
$|\bar C\cap L_\infty|\le 1\cdot 3=3$.

\textit{(c) Vertical tangencies, $|E_{\mathrm{vt}}|\le 6$.} For a
smooth point $p\in\bar C$ with coordinates $[X_0:Y_0:Z_0]$, the
projective tangent line $T_p\bar C$ is
\[
\bar F_X(p)\,X + \bar F_Y(p)\,Y + \bar F_Z(p)\,Z = 0,
\]
where subscripts denote partial derivatives. This line passes through
$p_\infty=[0:1:0]$ iff $\bar F_Y(p)=0$. Therefore
\[
E_{\mathrm{vt}}\;\subset\; \bar C(\R)\cap \{\bar F_Y=0\}.
\]
The polynomial $\bar F_Y$ has degree $\deg\bar F-1=2$. It shares no
common component with $\bar F$: such a common factor of positive degree
would be a divisor of $\bar F$ of degree at most $2$, contradicting
irreducibility of the cubic $\bar F$. B\'ezout then gives
$|\bar C\cap\{\bar F_Y=0\}|\le 2\cdot 3=6$, so $|E_{\mathrm{vt}}|\le 6$.

\textit{(d) Flexes, $|E_{\mathrm{fl}}|\le 3$.}  We use a standard real-cubic fact: an irreducible real projective cubic has at most three smooth real flexes.  In the smooth case this is classical, and Bonifant--Milnor prove the sharper statement that every smooth real cubic has exactly three real flexes~\cite[Corollary~6.5]{BonifantMilnor}.  In the singular irreducible case, the same upper bound follows from the two real normal forms, as follows.  For the cuspidal cubic
$\bar F = X^3 - Y^2Z$, a direct computation gives
$H_{\bar F} = -24XY^2$.  The intersection of $\{H_{\bar F}=0\}$ with
$\bar C(\R)$ consists of the cusp $[0:0:1]$ (singular) and the point
at infinity $[0:1:0]$ (smooth), giving at most one smooth real flex.
For the nodal cubic $\bar F = Y^2Z-X^2(X+Z)$ (acnodal form), a direct
computation gives $H_{\bar F} = 8(3XY^2+Y^2Z-X^2Z)$.  Restricting to
$\bar C$ and dehomogenizing $Z=1$ gives
$H_{\bar F}|_{\bar C} = 8X^3(3X+4)$, vanishing at $X=0$ (the node,
singular) and $X=-4/3$, where $Y^2 = -16/27 < 0$, so this branch
contributes no real points; the only smooth real flex is at infinity.
For the nodal cubic $\bar F = Y^2Z-X^2(X-Z)$ (crunodal form), the
analogous computation gives $H_{\bar F}|_{\bar C} = 8X^3(3X-4)$, with
smooth real flexes at $X=4/3$ (yielding two real points
$Y = \pm 4/(3\sqrt{3})$) and at infinity, totaling three.  In all
cases $|E_{\mathrm{fl}}|\le 3$.

Summing (overlaps only decrease the size of the union):
\[
|E| \;\le\; 1 + 3 + 6 + 3 \;=\; 13.
\]

\smallskip
\textbf{Step 2: Count the arcs.} The real locus $\bar C(\R)$ is a
closed subset of the compact projective plane $\RP^2$. Its complement
\[
E^c \;:=\; \bar C(\R)\setminus E
\]
is an open subset of $\bar C(\R)$ consisting of smooth affine points
of $\bar C$ (singular points are in $E_{\mathrm{sing}}$; points at
infinity are in $E_\infty$). In particular $E^c$ is a smooth real
$1$-manifold.

By Harnack's theorem for smooth real plane curves (see
\cite[Theorem~11.6.2]{BCR}), the number of connected components of the
real locus of a smooth real projective plane curve of degree $d$ is at
most $\binom{d-1}{2}+1$; for $d=3$ this gives at most $2$. The same
bound of $2$ connected components holds for singular irreducible real
cubics as well. Indeed, let $s$ be the unique singular point of $\bar C$.
Projection from $s$ to any real line not through $s$ identifies the
nonsingular real locus with $\RP^1$ minus the set of real tangent
directions at $s$, and there are at most two such directions. If there
are no real tangent directions, then $\bar C(\R)\setminus\{s\}$ is
connected and $s$ is isolated, giving two components in total. If there
is one real tangent direction, then $\bar C(\R)\setminus\{s\}$ is
connected and reinserting $s$ keeps the curve connected. If there are
two real tangent directions, then $\bar C(\R)\setminus\{s\}$ has two
components and reinserting $s$ joins them. Thus $\bar C(\R)$ has at
most two connected components in every singular case.

Thus $\bar C(\R)$ has at most $2$ connected components. To bound the
number of components of $E^c$, we split according to the presence of a
singular point. If $E_{\mathrm{sing}}=\varnothing$, then
$\bar C(\R)$ is already a smooth real $1$-manifold. In this case
$|E|\le 12$, and removing the $|E|$ smooth exceptional points increases
$\pi_0$ by at most $1$ per removal, so
\[
M=\pi_0(E^c)\le 2+|E|\le 14.
\]

If $|E_{\mathrm{sing}}|=1$, remove the singular point first. The local
possibilities are a figure-eight-type crunode, a cusp, or an acnode;
in these cases removal increases $\pi_0$ by $1$, increases $\pi_0$ by
$1$, or decreases $\pi_0$ by $1$, respectively. Hence
\[
\pi_0(\bar C(\R)\setminus E_{\mathrm{sing}})\le 3.
\]
The remaining exceptional points are smooth points of the smooth real
$1$-manifold $\bar C(\R)\setminus E_{\mathrm{sing}}$. Since there are
at most $|E|-1\le 12$ of them, removing them increases $\pi_0$ by at
most one each, and therefore
\[
M=\pi_0(E^c)\le 3+(|E|-1)\le 15.
\]
In either case $M\le 15$. Call these components $V_1,\ldots,V_M$.

Each $V_i$ is a connected open smooth $1$-manifold, so it is
homeomorphic either to an open interval or to a circle $S^1$. We
claim $V_i$ is not a circle. If it were, then the continuous
$x$-projection $\pi(x,y)=x$ would attain a maximum on the compact set
$V_i$, at some point $p\in V_i$; since $p$ is an interior point of a
smooth $1$-manifold, $d\pi_p=0$ at this maximum. But $d\pi$ is the
same as the $dx$ component of the tangent, which is nonzero whenever
the tangent is not vertical. Since $V_i\cap E_{\mathrm{vt}}=\varnothing$,
the tangent to $V_i$ is never vertical, so $d\pi\ne 0$ everywhere on
$V_i$. Contradiction. Therefore $V_i$ is homeomorphic to an open
interval.

In fact, since $d\pi\ne 0$ on $V_i$, the projection $\pi$ is a local
diffeomorphism; combined with $V_i$ being a topological interval, $\pi$
is a diffeomorphism of $V_i$ onto its image $I_i:=\pi(V_i)\subset\R$,
which is an open interval. So $V_i$ is the graph
$\{(x,f_i(x)):x\in I_i\}$ of a smooth function
$f_i=\mathrm{proj}_y\circ\pi^{-1}\colon I_i\to\R$. In particular $V_i$
is an open arc.

\smallskip
\textbf{Step 3: Each arc is a visibility patch.} Fix $i$ and write
$V=V_i$, $I=I_i$, $f=f_i$, so that $V=\{(x,f(x)):x\in I\}$. We show
that $V$ is a visibility patch of $\bar C$: for all distinct
$p,q\in V$, the open segment $(pq)^\circ$ contains no point of
$\bar C(\R)$.

\smallskip
\emph{Convexity of $f$.} Since $V$ contains no flexes, we claim
$f''(x)\ne 0$ on $I$. This is the standard affine interpretation of a
projective flex: for a smooth affine point $p=(x_0,f(x_0))$ of
$\bar C$ with nonvertical tangent, the vanishing of the projective
Hessian $H_{\bar F}(p)$ is equivalent to $p$ being a flex
\cite[Theorem~2.7]{BonifantMilnor}, which in graph coordinates is the
same as $f''(x_0)=0$. Thus $f''\ne 0$ on $I$. By continuity and
connectedness of $I$, the sign of $f''$ is constant on $I$; so $f$ is
strictly convex or strictly concave on $I$.

\smallskip
\emph{B\'ezout setup for chords.} Fix distinct $p,q\in V$, say with
$p=(s,f(s))$ and $q=(t,f(t))$, $s<t$. Let $\ell=\ell(s,t)$ be the
(projective) chord line through $p$ and $q$. Since $\bar C$ is
irreducible of degree $3$, $\ell\not\subset\bar C$, so by B\'ezout,
\[
\sum_{r\in\ell\cap\bar C}\mu_r(\ell\cap\bar C)=3,
\]
where $\mu_r$ denotes intersection multiplicity at $r\in\RP^2(\C)$.

At $p$, the tangent $T_p\bar C$ is the line $y-f(s)=f'(s)(x-s)$; its
slope is $f'(s)$. The chord $\ell$ has slope equal to the secant slope
\[
\sigma(s,t):=\frac{f(t)-f(s)}{t-s}.
\]
By the mean value theorem, $\sigma(s,t)=f'(c)$ for some
$c\in(s,t)$. Since $f''$ has constant nonzero sign on $I$ (shown
above), the derivative $f'$ is strictly monotonic on $I$; this
strict monotonicity, not the mean value theorem itself, yields
\[
\min\{f'(s),f'(t)\}<f'(c)<\max\{f'(s),f'(t)\},
\]
so in particular $\sigma(s,t)\ne f'(s)$ and $\sigma(s,t)\ne f'(t)$.
Hence $\ell$ is not tangent to $\bar C$ at $p$ or at $q$, and
\[
\mu_p(\ell\cap\bar C)=\mu_q(\ell\cap\bar C)=1.
\]

There is therefore a unique third point $r=r(s,t)\in
\ell\cap\bar C$ with $\mu_r(\ell\cap\bar C)=1$. A priori $r$ is
defined as a point of $\RP^2(\C)$.

\smallskip
\emph{$r$ is real.} Complex conjugation $\sigma_c$ acts on
$\ell(\C)\cap\bar C(\C)$ (both $\ell$ and $\bar C$ are defined over
$\R$). It fixes $p$ and $q$ (they are real) and permutes the three
(simple) intersection points as a multiset. So $\sigma_c$ maps
$\{p,q,r\}$ to itself. Since $p,q$ are fixed, $\sigma_c(r)\in
\{p,q,r\}$. If $\sigma_c(r)=p$ then $r=\sigma_c(p)=p$, contradicting
$r\ne p$; similarly $\sigma_c(r)\ne q$. So $\sigma_c(r)=r$, meaning
$r\in\RP^2(\R)$.

\smallskip
\emph{Continuous dependence of $r$ on $(s,t)$.} Define the open
region
\[
\Delta \;:=\; \{(s,t)\in I^2 : s<t\}\subset\R^2,
\]
which is convex (hence connected). Parametrize points of
the projective line $\ell(s,t)$ by $[\alpha:\beta]\in\RP^1$ via
\[
[\alpha:\beta]\;\longmapsto\;\alpha\,\widetilde p(s) + \beta\,\widetilde q(t)
\;\in\;\RP^2,
\]
where $\widetilde p(s),\widetilde q(t)\in\R^3\setminus\{0\}$ are
homogeneous representatives of $p(s)=(s,f(s),1)$ and
$q(t)=(t,f(t),1)$. Substituting into the homogeneous cubic $\bar F$
yields a homogeneous cubic
\[
\Phi_{s,t}(\alpha,\beta)\;=\;\bar F(\alpha\,\widetilde p(s) +\beta\,\widetilde q(t))
\]
in $(\alpha,\beta)$, whose coefficients depend smoothly on $(s,t)$.
Indeed, they are polynomial functions of $s$, $f(s)$, $t$, and $f(t)$,
and $f$ is smooth on $I$.

We argue the factorization explicitly, to make the continuity
transparent. Writing
\[
\Phi_{s,t}(\alpha,\beta)\;=\;A(s,t)\,\alpha^3 + B(s,t)\,\alpha^2\beta
+ C(s,t)\,\alpha\beta^2 + D(s,t)\,\beta^3,
\]
we have $A(s,t)=\bar F(\widetilde p(s))=0$ and
$D(s,t)=\bar F(\widetilde q(t))=0$ for all $(s,t)\in\Delta$ (since
$p(s),q(t)\in\bar C$). Hence
\[
\Phi_{s,t}(\alpha,\beta)\;=\;\alpha\beta\bigl(B(s,t)\,\alpha + C(s,t)\,\beta\bigr),
\]
and the three projective roots of $\Phi_{s,t}$ are $[1:0]$, $[0:1]$,
and $[-C(s,t):B(s,t)]$ (the last provided $(B,C)\ne(0,0)$). If
$(B,C)=(0,0)$ at some $(s,t)\in\Delta$, then $\Phi_{s,t}\equiv 0$,
meaning $\ell(s,t)\subset\bar C$; but $\ell(s,t)$ has degree $1$ and
$\bar C$ is irreducible of degree $3$, so this is impossible.
Therefore $[-C:B]\in\RP^1$ is well-defined and continuous in
$(s,t)$ on all of $\Delta$, and the corresponding point
$r(s,t)=-C(s,t)\,\widetilde p(s)+B(s,t)\,\widetilde q(t)\in\RP^2$
(up to scalar) is continuous on $\Delta$. No continuity beyond
$\Delta$ itself is claimed or used in the open/closed/empty
arguments below; in particular, no claim is made about the
tangential limit $t\searrow s$.

\smallskip
\emph{The key topological argument.} Define
\[
B\;:=\;\bigl\{(s,t)\in\Delta : r(s,t)\in \bigl(p(s)\,q(t)\bigr)^\circ\bigr\}.
\]
We show $B=\varnothing$.

\textit{$B$ is open in $\Delta$:} if $(s_0,t_0)\in B$, then
$r(s_0,t_0)$ lies in the interior of the chord, hence at positive
Euclidean distance from $\{p(s_0),q(t_0)\}$. By continuity of $r$,
$p$, and $q$ in $(s,t)$, this positive distance is preserved under
small perturbations of $(s_0,t_0)$, so a neighborhood of
$(s_0,t_0)$ in $\Delta$ lies in $B$.

\textit{$B$ is closed in $\Delta$:} let $(s_n,t_n)\to(s_0,t_0)$ in
$\Delta$ (so $s_0<t_0$) with $(s_n,t_n)\in B$. By continuity of $r$
on $\Delta$ (established above at the interior point $(s_0,t_0)$),
$r(s_n,t_n)\to r(s_0,t_0)$ in $\RP^2$; no subsequence or compactness
argument is needed, since $r$ is continuous. Each
$r(s_n,t_n)\in (p(s_n)q(t_n))^\circ\subset\R^2$, and the closed
chord segments $[p(s_n),q(t_n)]\subset\R^2$ converge (in Hausdorff
distance) to the bounded segment $[p(s_0),q(t_0)]\subset\R^2$; hence
the sequence $(r(s_n,t_n))_n$ is contained in a bounded region of
$\R^2$, and its limit $r(s_0,t_0)$, which a priori is a point of
$\RP^2$, is therefore a point of $\R^2$ lying in the closed segment
$[p(s_0),q(t_0)]$.

Suppose for contradiction $r(s_0,t_0)=p(s_0)$. By continuity of
$r$ on $\Delta$, the third root of $\Phi_{s,t}$ (in the
parametrization above) converges to $[1:0]$ as $(s,t)\to(s_0,t_0)$.
But $[1:0]$ is already a root of $\Phi_{s_0,t_0}$ (corresponding to
$p(s_0)$). Hence $\Phi_{s_0,t_0}$ has a root of multiplicity at
least $2$ at $[1:0]$, which means
$\mu_{p(s_0)}(\ell(s_0,t_0)\cap\bar C)\ge 2$, i.e.\ $\ell(s_0,t_0)$
is tangent to $\bar C$ at $p(s_0)$. But we showed above that for
any $(s,t)\in\Delta$, $\ell(s,t)$ is not tangent to $\bar C$ at
$p(s)$ or $q(t)$. Contradiction. Similarly $r(s_0,t_0)\ne q(t_0)$.
Hence $r(s_0,t_0)\in (p(s_0)q(t_0))^\circ$, so $(s_0,t_0)\in B$.

\textit{$B$ is empty near the diagonal} (local convexity argument).
Fix $s_0\in I$; we exhibit a neighborhood of $(s_0,s_0)$ in
$\bar\Delta$ (the closure of $\Delta$) whose intersection with
$\Delta$ is disjoint from $B$.

Since $V$ is open in $\bar C(\R)$ (as a connected component of
$\bar C(\R)\setminus E$), and $\bar C(\R)$ is closed in $\RP^2$, and
$p(s_0)\in V$ is disjoint from the closed set $\bar C(\R)\setminus V$,
there exists an open ball $U\subset\R^2$ centered at $p(s_0)$ such
that $U\cap\bar C(\R)\subset V$. By continuity of the
parametrization $x\mapsto(x,f(x))$, there exists $\epsilon>0$ such
that for all $s,t\in I$ with $|s-s_0|,|t-s_0|<\epsilon$, both
$p(s),q(t)\in U$, and since $U$ is convex, the closed segment
$[p(s),q(t)]\subset U$.

For such $(s,t)\in\Delta$ (i.e., also with $s<t$), we claim
$(p(s)q(t))^\circ\cap\bar C(\R)=\varnothing$. Indeed, any point of
$\bar C(\R)$ in this open segment lies in $U\cap\bar C(\R)\subset V$.
A point of $V$ in this open segment has $x$-coordinate in the open
interval $(s,t)$ (the $x$-range of the open chord segment). The
portion of $V$ with $x$-coordinate in $(s,t)$ is precisely the
graph $\{(x,f(x)):x\in(s,t)\}$. By strict convexity (or strict
concavity) of $f$ on $I\supset[s,t]$, this graph lies strictly
above (or strictly below) the chord over $(s,t)$, so no point of
this graph lies on the chord. Therefore
$(p(s)q(t))^\circ\cap V=\varnothing$, and hence
$(p(s)q(t))^\circ\cap\bar C(\R)=\varnothing$.

In particular, $r(s,t)\notin(p(s)q(t))^\circ$, so $(s,t)\notin B$.
Thus $B$ is disjoint from the neighborhood
$\{(s,t)\in\Delta:|s-s_0|,|t-s_0|<\epsilon\}$ of $(s_0,s_0)$.

\textit{Conclusion.} $B$ is a clopen subset of $\Delta$ that is not
all of $\Delta$ (it fails near the diagonal, by the local convexity
argument just given, and $\Delta$ accumulates on the diagonal). Since
$\Delta$ is connected (it is a convex open subset of $\R^2$),
$B=\varnothing$.

Therefore for all $p,q\in V$ the third intersection point $r=r(s,t)$
does not lie in $(pq)^\circ$. Combined with the fact that $p$ and $q$
are the only other intersections of $\ell$ with $\bar C$, we
conclude $(pq)^\circ\cap \bar C(\R)=\varnothing$: no real point of
$\bar C$ other than $p$ and $q$ lies on the line through $p,q$ inside
the segment. So $V$ is a visibility patch of $\bar C$.

\smallskip
\textbf{Step 4: No three patch points are collinear.} With Step~3
already established (i.e., for every $V=V_i$, $(pq)^\circ\cap
\bar C(\R)=\varnothing$ for all distinct $p,q\in V$), we now derive
the no-three-collinear property as a direct corollary. Suppose
$p,q,r$ are three distinct points of $V$ lying on a common line
$\ell$. Then $\ell\not\subset\bar C$ (degree reasons), and B\'ezout
gives $|\ell\cap\bar C|\le 3$, so $\ell\cap\bar C=\{p,q,r\}$. One of
$p,q,r$ lies strictly between the other two on $\ell$, say
$q\in(pr)^\circ$. Then $q\in V\subset\bar C(\R)$ and
$q\in(pr)^\circ$, giving $q\in(pr)^\circ\cap\bar C(\R)$, which
contradicts the visibility-patch property established in Step~3
applied to the distinct points $p,r\in V$. Hence no such triple
exists. (Note: Step~3 is proved independently of the
no-three-collinear claim, so this argument is not circular.)
\end{proof}

\begin{remark}
The constants $|E|\le 13$ and $M\le 15$ are not claimed to be sharp. A smooth real
cubic has exactly $3$ real flexes (the remaining $6$ of the nine
complex flexes come in complex-conjugate pairs; see
\cite[Cor.~6.5 and Theorem~2.10]{BonifantMilnor}). The vertical-tangency
estimate $|E_{\mathrm{vt}}|\le 6$ is the direct B\'ezout bound used
above, and this conservative uniform estimate is sufficient for the
applications below.
\end{remark}

\section{Visibility on a conic}\label{sec:conic}

The conic case is simpler because a line meets an irreducible conic in
at most two points.

\begin{lemma}\label{lem:conic-visibility}
Let $Q\subset\RP^2$ be an irreducible real projective conic. Then
$Q(\R)$ is a visibility patch of $Q$: for any two distinct points
$p,q\in Q(\R)$, the open segment $(pq)^\circ$ (in any affine chart
containing $p$ and $q$) contains no point of $Q(\R)$. Moreover, no
line meets $Q(\R)$ in three points.
\end{lemma}

\begin{proof}
If $Q(\R)=\varnothing$ the statement is vacuous, so assume
$Q(\R)\ne\varnothing$. Let $p,q\in Q(\R)$ be distinct. The chord line
$\ell$ through $p$ and $q$ is not contained in $Q$ (irreducibility of
$Q$, combined with $\deg\ell=1<2=\deg Q$). By B\'ezout,
$|\ell\cap Q|\le 2$, so $\ell\cap Q=\{p,q\}$. Hence
$(pq)^\circ\cap Q(\R)=\varnothing$.

The second statement is immediate: a line meets $Q$ in at most~$2$
points.
\end{proof}

\section{The cubic-container theorem}\label{sec:container-proof}

We now prove the deterministic container theorem. The statement is slightly more detailed than the numerical bound in the introduction: it gives a visible clique cover of the cubic-supported part of the set.

\begin{theorem}[Cubic-container theorem with clique covers]\label{thm:cubic-container}
Let $A\subset\R^2$ be finite, $|A|=n$, with no $k$ collinear points. Let $\Gamma\subset\RP^2$ be a real projective cubic and set
\[
  s:=|A\setminus \Gamma|,\qquad m:=|A\cap\Gamma|=n-s.
\]
Assume $m>3(k-1)$. Then $\Gamma$ is not a union of three real lines, and one of the following two alternatives holds.

\begin{enumerate}[label=\textup{(\roman*)}]
\item If $\Gamma=Q\cup L$, where $Q$ is an irreducible real conic and $L$ is a real line, and if $a:=|A\cap L|$, then
\[
  \vartheta_A(A\cap Q)\le s+a+1
\]
and hence
\[
  \omega(V(A))\ge \frac{|A\cap Q|}{s+a+1}
  \ge \frac{n-s-(k-1)}{s+k}.
\]
Moreover $A\cap\Gamma$ has a visible clique cover in $A$ with at most $s+2a+1\le s+2k-1$ parts.

\item If $\Gamma$ is irreducible, then $A\cap\Gamma$ has a visible clique cover in $A$ with at most
\[
  15(s+1)+13
\]
parts, and
\[
  \omega(V(A))\ge \frac{n-s-13}{15(s+1)}.
\]
\end{enumerate}
Consequently,
\begin{equation}\label{eq:cubic-container-bound}
\omega(V(A))\ge
\max\left\{1,
\min\left(
\frac{n-s-(k-1)}{s+k},
\frac{n-s-13}{15(s+1)}
\right)\right\}.
\end{equation}
\end{theorem}

\begin{proof}
Since $m=|A\cap\Gamma|>3(k-1)$, the cubic $\Gamma$ cannot be a union of three real lines: each affine line component contains at most $k-1$ points of $A$, and a line component equal to the line at infinity contains no point of $A$.

We distinguish the remaining two cases.

\medskip
\noindent\textbf{Case 1: $\Gamma=Q\cup L$, where $Q$ is an irreducible real conic and $L$ is a real line.}
Let $a:=|A\cap L|$. Since $a\le k-1$,
\[
  |A\cap Q|\ge n-s-a\ge n-s-(k-1).
\]
The subcase $Q(\R)=\varnothing$ is impossible under the hypothesis $m>3(k-1)$, because then $A\cap Q=\varnothing$ and $m\le |A\cap L|\le k-1$. Hence $Q(\R)\ne\varnothing$.

Let $V=Q(\R)\cap\R^2$ and $X=A\cap V$. By Lemma~\ref{lem:conic-visibility}, no line meets $V$ in three points. A blocker of a pair from $X$ cannot lie on $Q$, because a line meets the irreducible conic $Q$ in at most two points. Hence all blockers of pairs from $X$ lie either on $L$ or outside $\Gamma$. Thus the blocker set $B$ in Lemma~\ref{lem:blocker-colouring} has size
\[
  b\le |A\cap L|+|A\setminus\Gamma|=a+s.
\]
The lemma gives
\[
  \vartheta_A(A\cap Q)\le s+a+1,
  \qquad
  \omega(V(A))\ge \frac{|A\cap Q|}{s+a+1}
  \ge \frac{n-s-(k-1)}{s+k}.
\]
To cover all of $A\cap\Gamma$, add the points of $(A\cap L)\setminus Q$ as singletons. This uses at most $a$ further parts, so
\[
  \vartheta_A(A\cap\Gamma)\le s+2a+1\le s+2k-1.
\]

\medskip
\noindent\textbf{Case 2: $\Gamma$ is irreducible.}
Apply Lemma~\ref{lem:cubic-visibility} in the affine chart containing $A$. We obtain a finite exceptional set $E\subset\Gamma(\R)$ with $|E|\le 13$ and visibility patches $V_1,\ldots,V_M\subset\R^2$ with $M\le 15$ such that
\[
  \Gamma(\R)=E\sqcup V_1\sqcup\cdots\sqcup V_M.
\]
For each $i$, put $X_i=A\cap V_i$. If $p,q\in X_i$, then the visibility-patch property gives
\[
  (pq)^\circ\cap \Gamma(\R)=\varnothing.
\]
Therefore every blocker of a pair from $X_i$ lies outside $\Gamma$, and the blocker set in Lemma~\ref{lem:blocker-colouring} has size at most $s$. Since no line meets $V_i$ in three points, the lemma gives
\[
  \vartheta_A(X_i)\le s+1
\]
for every $i$.

Adding the points of $A\cap E$ as singletons, we get
\[
  \vartheta_A(A\cap\Gamma)\le M(s+1)+|A\cap E|\le 15(s+1)+13.
\]
Also, by pigeonholing, some patch satisfies
\[
  |A\cap V_i|\ge \frac{|A\cap\Gamma|-|A\cap E|}{M}
  \ge \frac{n-s-13}{15}.
\]
Applying Lemma~\ref{lem:blocker-colouring} to that patch gives
\[
  \omega(V(A))\ge \frac{n-s-13}{15(s+1)}.
\]

The two alternatives imply the numerical bound \eqref{eq:cubic-container-bound}. The maximum with $1$ only records the trivial nonempty-clique lower bound in small cases where the displayed fractions are less than one.
\end{proof}

\section{Point sets with few ordinary lines}\label{sec:few-ordinary-proof}

The Green--Tao theorem now upgrades the deterministic cubic-container theorem to a structural result for point sets with few ordinary lines.

\begin{theorem}[Few ordinary lines: bounded clique cover and linear clique]\label{thm:few-ordinary-linear}
For every integer $k\ge 2$ and every $K>0$ there are constants $c=c(k,K)>0$, $C_{\rm cov}=C_{\rm cov}(k,K)$ and $n_0=n_0(k,K)$ such that the following holds. If $A\subset\R^2$ has $n\ge n_0$ points, no $k$ collinear points, and at most $Kn$ ordinary lines, then there is a subset $A_0\subset A$ with
\[
  |A\setminus A_0|\le c_{\rm GT}K
\]
which has a visible clique cover in $A$ with at most $C_{\rm cov}$ parts. In particular,
\[
  \omega(V(A))\ge c n.
\]
More explicitly, for all sufficiently large $n$,
\begin{equation}\label{eq:explicit-main-bound}
\omega(V(A))\ge
\max\left\{1,
\min\left(
\frac{n-c_{\rm GT}K-(k-1)}{c_{\rm GT}K+k},
\frac{n-c_{\rm GT}K-13}{15(c_{\rm GT}K+1)}
\right)\right\}.
\end{equation}
\end{theorem}

\begin{proof}
Let $A\subset\R^2$ have $n$ points, no $k$ collinear points, and at most $Kn$ ordinary lines. Assume $n\ge n_{\rm GT}(K)$. By Theorem~\ref{thm:GT}, there is a real cubic $\Gamma\subset\RP^2$ with
\[
  s:=|A\setminus\Gamma|\le c_{\rm GT}K.
\]
Let $A_0:=A\cap\Gamma$. If $n-s\le 3(k-1)$, then $n\le c_{\rm GT}K+3(k-1)$, which is excluded once $n$ is larger than this constant. Hence Theorem~\ref{thm:cubic-container} applies.

The clique-cover assertion follows from the two alternatives in Theorem~\ref{thm:cubic-container}. For example, one may take
\[
C_{\rm cov}(k,K)=
\max\{\lceil c_{\rm GT}K\rceil+2k-1,
15(\lceil c_{\rm GT}K\rceil+1)+13\}.
\]
The explicit clique lower bound \eqref{eq:explicit-main-bound} follows by substituting $s\le c_{\rm GT}K$ into \eqref{eq:cubic-container-bound}. For fixed $k$ and $K$, the right-hand side is $\Omega_{k,K}(n)$.
\end{proof}

\begin{corollary}[Few-ordinary-lines version of Big-Line--Big-Clique]\label{cor:fixed-ell}
For every $k,\ell\ge 2$ and every $K>0$ there exists $n_0(k,\ell,K)$ such that every finite set $A\subset\R^2$ with $|A|\ge n_0(k,\ell,K)$, no $k$ collinear points, and at most $Kn$ ordinary lines contains $\ell$ mutually visible points.
\end{corollary}

\begin{proof}
Choose $n_0$ large enough that the explicit lower bound \eqref{eq:explicit-main-bound} is at least $\ell$, and also large enough for Theorem~\ref{thm:GT} to apply.
\end{proof}

\section[Calibration: the (4,4) case via cubic containers]{Calibration: the \texorpdfstring{$(4,4)$}{(4,4)} case via cubic containers}\label{sec:k4l4}

The ordinary-line hypothesis is substantive in the first nontrivial no-four-collinear case.  The following proposition is not new as a case of the known $\ell\le 5$ Big-Line--Big-Clique theorem, but it is a useful calibration: in the $(k,\ell)=(4,4)$ boundary, the absence of a visible $K_4$ itself forces the few-ordinary-lines regime.

\begin{proposition}[No four collinear and no visible four force few ordinary lines]\label{prop:44-few-ordinary}
Let $A\subset\R^2$ have $n$ points and no four collinear points.  If $V(A)$ has no $K_4$, then $A$ spans at most $n$ ordinary lines.
\end{proposition}

\begin{proof}
Let $t_2$ be the number of ordinary lines and $t_3$ the number of lines containing exactly three points of $A$.  Since no four points are collinear, every line determined by two points of $A$ is counted by either $t_2$ or $t_3$, and
\[
  \binom n2=t_2+3t_3.
\]
The visibility graph has exactly
\[
  e(V(A))=t_2+2t_3
\]
edges: an ordinary line contributes its unique pair, while a three-point line contributes exactly the two adjacent visible pairs along the line.  If $V(A)$ has no $K_4$, Tur\'an's theorem gives
\[
  e(V(A))\le \left(1-\frac13\right)\frac{n^2}{2}=\frac{n^2}{3}.
\]
Eliminating $t_3$ from the two displayed identities gives
\[
  t_2=3e(V(A))-2\binom n2\le n^2-n(n-1)=n.
\]
\end{proof}

\begin{corollary}[The $(4,4)$ case via cubic containers]\label{cor:44}
There is an absolute $n_0$ such that every $A\subset\R^2$ with $|A|\ge n_0$ contains either four collinear points or four mutually visible points.
\end{corollary}

\begin{proof}
Assume that $A$ has no four collinear points and no four mutually visible points.  By Proposition~\ref{prop:44-few-ordinary}, $A$ has at most $n$ ordinary lines.  Applying Theorem~\ref{thm:few-ordinary-linear} with $k=4$ and $K=1$ gives a visible clique of size at least $c n$ for all sufficiently large $n$, contradiction once $n\ge 4/c$ and $n\ge n_0(4,1)$.
\end{proof}

\section{A bounded-degree algebraic-curve case}\label{sec:bounded-degree}

The cubic-container theorem also combines naturally with the Elekes--Szab\'o theorem on triple lines on algebraic curves. This gives a second, independent structured regime in which the Big-Line--Big-Clique conclusion follows.

We use the following consequence of Elekes and Szab\'o's orchard theorem. For fixed degree $d$, if an irreducible algebraic curve of degree at most $d$ supports a set of $n$ points determining quadratically many distinct $3$-rich lines, then, for sufficiently large $n$, the curve must be cubic. More precisely, this is the special case $\Gamma_1=\Gamma_2=\Gamma_3$ of \cite[Theorem~4.1]{ElekesSzabo2024}; see also Remark~4.2 of that paper for the reducible-container interpretation.

\begin{theorem}[Irreducible bounded-degree curve case]\label{thm:bounded-degree-curve}
For every $d\ge 1$ and every $k,\ell\ge2$ there exists $n_0=n_0(d,k,\ell)$ such that the following holds. Let $C\subset\R^2$ be an irreducible real algebraic curve of degree at most $d$, and let $A\subset C(\R)$ be finite with $|A|\ge n_0$. Then $A$ contains either $k$ collinear points or $\ell$ mutually visible points.
\end{theorem}

\begin{proof}
Assume, for contradiction, that $A\subset C(\R)$ has $n$ points, no $k$ collinear points and no $\ell$ mutually visible points. We take $n$ sufficiently large as a function of $d,k,\ell$.

If $k\le3$, then no $k$ collinear points implies in particular that no three points of $A$ are collinear. Thus every pair of points of $A$ is visible in $A$, so $n<\ell$, impossible for large $n$. Hence, in the counting argument below, we may assume $k\ge4$.

If $C$ is a line, then $n\ge k$ already gives $k$ collinear points, so this case is impossible under the assumption. If $C$ is an irreducible conic, Lemma~\ref{lem:conic-visibility} implies that every two points of $A$ are visible in $A$, so $A$ itself is a visible clique. Thus $n<\ell$, again impossible for large $n$.

If $C$ is an irreducible cubic, apply Theorem~\ref{thm:cubic-container} with $\Gamma$ the projective closure of $C$ and $s=0$. The irreducible-cubic alternative gives
\[
  \omega(V(A))\ge \frac{n-13}{15},
\]
which is at least $\ell$ for large $n$. Hence this case is also impossible under the assumption.

It remains to consider $\deg C\ge4$. Since $V(A)$ has no $K_\ell$, Tur\'an's theorem gives at least
\[
  \frac{n^2}{2(\ell-1)}-\frac n2
\]
non-visible pairs. For $n\ge4(\ell-1)$, this is at least $n^2/(4(\ell-1))$. Each non-visible pair lies on a line containing a third point of $A$. Since no line contains $k$ points of $A$, each $3$-rich line contains at most $k-1$ points and hence accounts for at most $\binom{k-1}{2}$ such pairs. Therefore $A$ determines at least
\[
  c_{k,\ell} n^2,
  \qquad
  c_{k,\ell}:=\frac{1}{4(\ell-1)\binom{k-1}{2}},
\]
distinct $3$-rich lines.

Apply Elekes--Szab\'o \cite[Theorem~4.1]{ElekesSzabo2024} with $\Gamma_1=\Gamma_2=\Gamma_3=C$ and $H_1=H_2=H_3=A$. For the constant $c_{k,\ell}$ and degree $d$, that theorem supplies $\eta=\eta(c_{k,\ell},d)>0$ and $n_{\rm ES}=n_{\rm ES}(c_{k,\ell},d)$ such that, once $n\ge n_{\rm ES}$, the existence of at least $c_{k,\ell}n^2\ge c_{k,\ell}n^{2-\eta}$ distinct triple lines forces $C\cup C\cup C=C$ to be a cubic. This contradicts $\deg C\ge4$.

Thus the simultaneous failure of both conclusions is impossible for all sufficiently large $n$.
\end{proof}

\begin{theorem}[Robust low-degree algebraic containers]\label{thm:robust-low-degree}
Fix integers $d\ge1$, $k\ge2$, and $\ell\ge2$.  There are constants $\varepsilon=\varepsilon(d,k,\ell)>0$ and $N_0=N_0(d,k,\ell)$ with the following property.  Let $C\subset\R^2$ be an irreducible real algebraic curve of degree at most $d$, and let $A\subset\R^2$ be finite with no $k$ collinear points.  Put $X=A\cap C$ and $N=|X|$.  If
\[
  N\ge N_0,
  \qquad |A\setminus C|\le \varepsilon N,
\]
then $A$ contains $\ell$ mutually visible points.
\end{theorem}

\begin{proof}
Choose $\varepsilon>0$ small enough to satisfy all inequalities imposed below, and take $N_0\ge \max\{k,3(k-1)+1\}$ sufficiently large in terms of $d,k,\ell$ and the Elekes--Szab\'o threshold.

If $\deg C=1$, then $N\ge N_0\ge k$ contradicts the no-$k$-collinear hypothesis.  If $\deg C=2$, then no line meets $C$ in three points.  Applying Lemma~\ref{lem:blocker-colouring} to $V=C(\R)$ gives
\[
  \omega(V(A))\ge \frac{N}{|A\setminus C|+1}\ge \frac{N}{\varepsilon N+1},
\]
which is at least $\ell$ when $\varepsilon<1/(2\ell)$ and $N_0$ is large.

If $\deg C=3$, then $N\ge N_0>3(k-1)$, so Theorem~\ref{thm:cubic-container} applies with $s=|A\setminus C|$ and $n=N+s$.  The irreducible-cubic term gives
\[
  \omega(V(A))\ge \frac{N-13}{15(s+1)}\ge \ell
\]
for $\varepsilon<1/(30\ell)$ and $N_0$ large.

It remains to handle $4\le \deg C\le d$.  Suppose, for contradiction, that $A$ has no $\ell$ mutually visible points.  Then the ambient visibility graph induced on $X$ has no $K_\ell$.  By Tur\'an's theorem, for $N$ large the number of non-visible pairs in $X$ is at least
\[
  \frac{N^2}{4(\ell-1)}.
\]
A non-visible pair in $X$ is blocked either by a point of $X$ or by a point of $A\setminus C$.  Fix one external point $r\in A\setminus C$.  Since $C$ is irreducible and not a line, every line through $r$ meets $C$ in at most $\deg C\le d$ points.  Partitioning $X$ by the lines through $r$, the number of pairs of $X$ blocked by $r$ is at most
\[
  \frac{d-1}{2}N.
\]
Thus all external blockers together block at most
\[
  |A\setminus C|\frac{d-1}{2}N\le \varepsilon\frac{d-1}{2}N^2
\]
pairs from $X$.  Taking
\[
  \varepsilon<\frac{1}{4(d-1)(\ell-1)}
\]
leaves at least $N^2/(8(\ell-1))$ non-visible pairs whose blocker lies in $X$.

Every such internally blocked pair lies on a line containing at least three points of $X$.  Since a line meets $C$ in at most $d$ points, each $3$-rich line of $X$ accounts for at most $\binom d2$ internally blocked pairs.  Therefore $X$ determines at least
\[
  c_{d,\ell}N^2,
  \qquad
  c_{d,\ell}:=\frac{1}{8(\ell-1)\binom d2},
\]
distinct $3$-rich lines.  Applying Elekes--Szab\'o~\cite[Theorem~4.1]{ElekesSzabo2024} with $\Gamma_1=\Gamma_2=\Gamma_3=C$ and $H_1=H_2=H_3=X$, we conclude, once $N\ge N_0$, that $C\cup C\cup C=C$ is a cubic.  This contradicts $\deg C\ge4$.
\end{proof}

\begin{remark}
Theorem~\ref{thm:robust-low-degree} is the form in which the algebraic-curve case is most useful for visibility: it is ambient, not induced.  It allows a small but positive proportion of off-curve blockers, and the proof shows exactly where the cubic-container theorem becomes necessary.

We record Theorem~\ref{thm:robust-low-degree} as an independent
strengthening of Theorem~\ref{thm:bounded-degree-curve}; while
not used elsewhere in the paper, it is the natural ambient analogue
that allows a small proportion of off-curve points.
\end{remark}

\section{Sharpness and limitations of the blocker mechanism}\label{sec:sharpness}

The blocker-colouring lemma is not merely a proof device; its dependence on the number of external blockers is forced already in the simplest conic example.

\begin{proposition}[Sharpness for one blocker]\label{prop:one-blocker-sharp}
For every even integer $m\ge2$ there is a set $A\subset\R^2$ and an irreducible conic $Q$ such that $|A\cap Q|=m$, $|A\setminus Q|=1$, no four points of $A$ are collinear, and the largest mutually visible subset of $A\cap Q$ in the ambient set $A$ has size exactly $m/2$.
\end{proposition}

\begin{proof}
Let $Q$ be the unit circle and let $o=(0,0)$. Choose $m/2$ antipodal pairs on $Q$, with no pair repeated, and let $X$ be the set of their $m$ endpoints. Put $A=X\cup\{o\}$.

The only collinear triples in $A$ are $o$ together with one antipodal pair, so no four points are collinear. In $X$, the two endpoints of each antipodal pair are not visible in $A$, because $o$ lies on the open segment between them. Any two non-antipodal points of $X$ are visible: their chord does not pass through $o$, and a line meets the circle in at most two points. Therefore the non-visibility graph on $X$ is a perfect matching. A visible clique in $X$ can choose at most one endpoint from each antipodal pair, so it has size at most $m/2$; choosing one endpoint from every antipodal pair gives a visible clique of size $m/2$.
\end{proof}

This example shows that the denominator $s+1$ in the local conic and cubic-patch bounds cannot be removed. It also explains why a merely positive-density cubic conclusion is insufficient for the full Big-Line--Big-Clique Conjecture: if a linear number of ambient blockers remains outside the cubic, the local lemma alone cannot force an arbitrarily prescribed fixed visible clique.

\section{The dense-orchard core forced by no visible clique}\label{sec:orchard}

This section isolates a structural obstacle to removing the ordinary-line hypothesis.  The result below is elementary but useful: the absence of a visible clique produces not just many collinear triples, but a positive-density subconfiguration in which every point is incident to linearly many $3$-rich lines.

A line is called \emph{$3$-rich} for a set $A$ if it contains at least three points of $A$.

\begin{proposition}[Dense-orchard core]\label{prop:orchard-core}
Fix $k\ge 4$ and $\ell\ge 2$.  Let $A\subset\R^2$ have $n$ points, no $k$ collinear points, and no $\ell$ mutually visible points.  If $n\ge 4(\ell-1)$, then there is a subset $A'\subset A$ with
\[
  |A'|\ge \frac{n}{8(\ell-1)\binom{k-2}{2}}
\]
such that every point of $A'$ is incident to at least
\[
  \frac{n}{24(\ell-1)\binom{k-2}{2}}
\]
distinct lines that contain at least three points of $A'$.
\end{proposition}

\begin{proof}
Since $V(A)$ has no $K_\ell$, Tur\'an's theorem gives
\[
  e(V(A))\le \left(1-\frac1{\ell-1}\right)\frac{n^2}{2}.
\]
Thus the number $M$ of non-visible pairs is at least
\[
  \binom n2-\left(1-\frac1{\ell-1}\right)\frac{n^2}{2}
  =\frac{n^2}{2(\ell-1)}-\frac n2
  \ge \frac{n^2}{4(\ell-1)},
\]
where the final inequality follows from $n\ge 4(\ell-1)$.

Every non-visible pair lies on a line containing at least three points of $A$.  If a line contains $t\ge3$ points of $A$, it contributes at most $\binom t2$ such pairs and exactly $\binom t3$ collinear triples.  Since $\binom t2\le 3\binom t3$ for every $t\ge3$, the number $T$ of unordered collinear triples in $A$ satisfies
\[
  T\ge M/3\ge \frac{n^2}{12(\ell-1)}.
\]
Set
\[
  \delta:=\frac{1}{12(\ell-1)},\qquad D_k:=\binom{k-2}{2}.
\]
Consider the $3$-uniform hypergraph $\mathcal H$ whose vertices are the points of $A$ and whose edges are the collinear triples.  Then $e(\mathcal H)\ge \delta n^2$.

Prune vertices as follows.  Repeatedly delete any vertex whose current degree in the remaining hypergraph is less than $\delta n/2$.  Each deletion removes fewer than $\delta n/2$ hyperedges, so after at most $n$ deletions fewer than $\delta n^2/2$ hyperedges have been removed.  The remaining hypergraph $\mathcal H'$ therefore has at least $\delta n^2/2$ edges.  Let $A'$ be its vertex set.  By construction, every point of $A'$ lies in at least $\delta n/2$ collinear triples supported entirely in $A'$.

A fixed point $p\in A'$ participates in at most $D_k$ collinear triples on any single line through $p$, since each such line contains at most $k-1$ points of $A$, hence at most $k-2$ other points besides $p$.  Therefore $p$ is incident to at least
\[
  \frac{\delta n/2}{D_k}=\frac{n}{24(\ell-1)D_k}
\]
distinct $3$-rich lines for $A'$.

It remains to lower-bound $|A'|$.  The degree of any point in the original triple hypergraph is at most $D_k n$: there are at most $n-1$ lines through $p$ determined by another point, and each contributes at most $D_k$ triples through $p$.  Hence
\[
  e(\mathcal H')\le \frac13 |A'|D_k n.
\]
Together with $e(\mathcal H')\ge\delta n^2/2$, this gives
\[
  |A'|\ge \frac{3\delta}{2D_k}n
  =\frac{n}{8(\ell-1)D_k},
\]
as claimed.
\end{proof}

\begin{theorem}[Dense orchard cores on bounded-degree curves collapse to cubics]\label{thm:orchard-curve-collapse}
Fix $d\ge1$, $k\ge4$, and $\rho>0$.  There is $n_0=n_0(d,k,\rho)$ such that the following holds.  Let $C\subset\R^2$ be an irreducible real algebraic curve of degree at most $d$, and let $A\subset C(\R)$ have $n\ge n_0$ points with no $k$ collinear points.  If every point of $A$ is incident to at least $\rho n$ distinct $3$-rich lines of $A$, then $C$ is a cubic.
\end{theorem}

\begin{proof}
If $C$ is a line, then $n<k$, which is impossible for $n_0\ge k$.  Thus $C$ is not a line, and every line meets $C$ in at most $d$ points by B\'ezout.  The point--rich-line incidence count is at least $\rho n^2$.  Each rich line is incident to at most $d$ points of $A$, so the number of distinct $3$-rich lines is at least $(\rho/d)n^2$.

Apply Elekes--Szab\'o~\cite[Theorem~4.1]{ElekesSzabo2024} with $\Gamma_1=\Gamma_2=\Gamma_3=C$ and $H_1=H_2=H_3=A$.  For the constant $c=\rho/d$ and degree $d$, that theorem supplies $\eta=\eta(c,d)>0$ and $n_0(c,d)$ such that $(\rho/d)n^2\ge c n^{2-\eta}$ rich triple lines force $C\cup C\cup C=C$ to be a cubic.
\end{proof}

\begin{corollary}[Concrete use of the dense-orchard core]\label{cor:orchard-core-used}
Fix $d\ge1$, $k\ge4$, and $\ell\ge2$.  There is $n_0=n_0(d,k,\ell)$ such that if $C\subset\R^2$ is an irreducible real algebraic curve of degree at most $d$ and $A\subset C(\R)$ has $n\ge n_0$ points with no $k$ collinear points and no $\ell$ mutually visible points, then $C$ is a cubic.
\end{corollary}

\begin{proof}
Apply Proposition~\ref{prop:orchard-core} to obtain $A'\subset A$ of size at least $\alpha n$, where $\alpha=1/(8(\ell-1)\binom{k-2}{2})$, such that every point of $A'$ is incident to at least $\gamma n$ distinct $3$-rich lines of $A'$, where $\gamma=1/(24(\ell-1)\binom{k-2}{2})$.  Since $|A'|\le n$, every point of $A'$ is incident to at least $\gamma |A'|$ such lines.  Applying Theorem~\ref{thm:orchard-curve-collapse} to $A'\subset C(\R)$ gives that $C$ is cubic.
\end{proof}

\begin{remark}
Proposition~\ref{prop:orchard-core} is not a proof of the full Big-Line--Big-Clique Conjecture.  It identifies the next missing theorem.  A positive solution to the following ambient container problem, with a sufficiently strong blocker conclusion, would combine with Theorem~\ref{thm:cubic-container} to settle the conjecture.
\end{remark}

\begin{conjecture}[Ambient cubic-container problem]\label{conj:ambient-container}
For every $k,\ell$ there are constants $\alpha>0$ and $\beta<1/(\ell-1)$ such that the following holds. If $A\subset\R^2$ has $n$ points, no $k$ collinear points, and no $\ell$ mutually visible points, then there is a real cubic $\Gamma$ and a visibility patch $W\subset\Gamma(\R)$ for which
\[
  |A\cap W|\ge \alpha n
\]
and the number of ambient blockers in $A\setminus\Gamma$ for pairs from $A\cap W$ is at most
\[
  \beta |A\cap W|.
\]
\end{conjecture}

\begin{remark}[Ambient blockers and the threshold]
The statement above is intentionally stronger than the bare positive-density cubic question: merely placing $\alpha n$ points on a cubic is not enough by itself, because the remaining $(1-\alpha)n$ points may act as external blockers.  The cubic must be an \emph{ambient} container, not just an induced-subset container.

The threshold $\beta<1/(\ell-1)$ is the one naturally suggested by
Lemma~\ref{lem:blocker-colouring}. If $X=A\cap W$ has $m$ points and at
most $\beta m$ ambient blockers outside $\Gamma$, then the patch
property and Lemma~\ref{lem:blocker-colouring} give a visible subset
of size at least
\[
  \frac{m}{\beta m+1}.
\]
As $m$ grows, this lower bound tends to $1/\beta$. Thus
$\beta<1/(\ell-1)$ forces, for sufficiently large $m$, a visible clique
of size at least $\ell$ on the patch.
\end{remark}

\section{Concluding remarks}\label{sec:remarks}

\bigskip
\noindent{\bfseries Quantitative form.}\par\nobreak\smallskip
The only large threshold in Theorem~\ref{thm:few-ordinary-linear} is the Green--Tao threshold
\[
n_{\rm GT}(K)\le\exp\exp(CK^C).
\]
This doubly exponential dependence is inherited entirely from the Green--Tao structure theorem; any quantitative improvement there would immediately improve the threshold in Theorem~\ref{thm:few-ordinary-linear}.

Once the cubic container is available, the visibility part is linear and explicit.  An irreducible cubic with $s$ off-cubic points gives a clique of size at least
\[
  \frac{n-s-13}{15(s+1)},
\]
and a line-plus-conic cubic gives a clique of size at least
\[
  \frac{n-s-(k-1)}{s+k}.
\]

\bigskip
\noindent{\bfseries Why the theorem is naturally cubic.}\par\nobreak\smallskip
The Green--Tao theorem shows that configurations with very few ordinary lines are governed by cubics.  The local proof here explains why cubics are also compatible with visibility: after removing finitely many flexes, points at infinity, vertical tangencies and singularities, an irreducible real cubic decomposes into arcs whose chords are empty with respect to the cubic itself.  Ambient blockers are then the only remaining obstruction.

\bigskip
\noindent{\bfseries Limits of the method.}\par\nobreak\smallskip
The full Big-Line--Big-Clique Conjecture would require converting the dense-orchard core in Proposition~\ref{prop:orchard-core} into an ambient cubic container.  Ordinary incidence estimates alone do not accomplish this: they allow $\Theta(n^2)$ incidences between $n$ points and $\Theta(n^2)$ constant-rich lines without forcing a single low-degree curve.  Any proof of the full conjecture through the present framework must therefore add a genuinely algebraic structural theorem.

\section*{Acknowledgements and AI-use disclosure}

AI-assisted tools were used for LaTeX typesetting, literature search, and proofreading. The author verified all mathematical statements, proofs, references, and final wording.

\end{document}